\documentclass{amsart}
\usepackage{hyperref,enumerate}
\usepackage{xcolor}
\usepackage[utf8]{inputenc}

\newtheorem{theorem}{Theorem}[section]

\newtheorem{corollary}[theorem]{Corollary}

\theoremstyle{definition}
\newtheorem{definition}[theorem]{Definition}

\theoremstyle{remark}

\numberwithin{equation}{section}

\title[Extremal Metrics]{Extremal Eigenvalues Of The Conformal Laplacian Under Sire-Xu Normalization}
\author{Samuel P\'{e}rez-Ayala}
\address{Department of Mathematics \\
         University of Notre Dame\\
         Notre Dame, IN 46556}

\email{\href{sperezay@nd.edu}{sperezay@nd.edu}}

\begin{document}

%%%%%%%%%%%%%%%%
\begin{abstract}
Let $(M^n,g)$ be a closed Riemannian manifold of dimension $n\ge 3$. We study the variational properties of the $k$-th eigenvalue functional $\tilde g\in[g] \mapsto \lambda_k(L_{\tilde g})$ under a non-volume normalization proposed by Sire-Xu. We discuss necessary conditions for the existence of extremal eigenvalues under such normalization. Also, we discuss the general existence problem when $k=1$.
\end{abstract}
%%%%%%%%%%%%%%%%

\maketitle

%%%%%%%%%%%%%%%%%%%%%%%%%%%%%%%%%%%%%%%%
\section{Introduction}\label{Introduction}
%%%%%%%%%%%%%%%%%%%%%%%%%%%%%%%%%

Let $(M^n,g)$ be a closed (compact, no boundary) Riemannian manifold of dimension $n\ge 3$, equipped with a smooth Riemannian metric $g$. We consider the metric dependent operator known as the Conformal Laplacian defined by
\begin{equation}
L_g := -\Delta_g + c_nR_g.
\end{equation}
Here $c_n= \frac{n-2}{4(n-1)}$ and $\Delta_g$ is the Laplace-Beltrami operator defined as a negative operator. $L_g$ is conformally covariant in the following sense: if $\tilde g = \mu^{4/(n-2)}g$, then for any $u\in C^\infty(M^n)$ we have
\begin{equation}\label{ConfTrans}
L_{\mu^{\frac{4}{n-2}}g} (u) = \mu^{-\frac{n+2}{n-2}}L_g(\mu u).
\end{equation}
Since $M^n$ is compact, $L_g$ has a discrete spectrum and we denote it by
\begin{equation}
\lambda_1(L_g)<\lambda_2(L_g) \le \lambda_3(L_g) \le \cdots \le \lambda_k(L_g) \rightarrow \infty,
\end{equation}
where each eigenvalue is repeated according to their multiplicity. Notice that the first eigenvalue is always simple. 

%There are various conformally invariant properties and quantities related with the Conformal Laplacian. The  operator $L_g$ is positive if and only if the Yamabe invariant $Y(M^n,[g])$ is positive. Specifically, the sign of $\lambda_1(L_g)$ coincides with the sign of $Y(M^n,[g])$. Another conformally invariant property is the dimension of the kernel of $L_g$, that is, the nullspace of $L_g$. This can be deduce from the transformation law \ref{ConfTrans}. Finally, the number of negative eigenvalues of $L_g$, which we denote by $\nu([g])$, is also a conformal invariant. This follows from the invariance of the nullspace and the continuity of $\lambda_k(L_{.})$ in the $C^\infty$-topology of Riemannian metrics (\cite{Canzani2}).

%We assume that $\lambda_0(L_g)>0$, that is, $L_g$ is a positive operator. This assumption is conformally invariant. In fact, the Yamabe invariant $Y(M,[g])$ is positive if and only if $\lambda_0(L_g)$ is positive. This could also be seen by looking at the corresponding energy and its conformal properties:
%\begin{equation}
%\int_M uL_{\mu^{\frac{4}{n-2}}g}u\;dv_{\mu^{\frac{4}{n-2}}g} = \int_M (u\mu)L_g(u\mu)\;dv_g.
%\end{equation}

There are many conformally invariant quantities associated with $L_g$. First, the sign of $\lambda_1(L_g)$ is conformally invariant. Indeed, the sign of the Yamabe invariant $Y(M^n,[g])$,
\begin{equation}
Y(M^n,[g]):= \inf_{u\in W^{1,2}(M^n,g)}\frac{\int_M uL_gu\;dv_g}{\left(\int_M|u|^{\frac{2n}{n-2}}\;dv_g\right)^{\frac{n-2}{n}}},
\end{equation}
coincides with the sign of $\lambda_1(L_g)$ (\cite{Kazdan}). Secondly, the dimension of $\text{Ker}(L_g)$ is also conformally invariant. This follows from the transformation law (\ref{ConfTrans}).

In this paper we will investigate the variational properties of the eigenvalue functional 
\begin{equation}\label{EigenFunctional}
\tilde g \in [g] \longmapsto \lambda_k(L_{\tilde g})
\end{equation}
under a non-standard normalization proposed by Sire-Xu in \cite{Sire}. Notice that for any positive real number $c>0$, the $k$-th eigenvalue scales as $\lambda_k(L_{cg}) = c^{-1}\lambda_k(L_g)$. This can be observed from (\ref{ConfTrans}). Therefore, in order to study possible ``critical points" of this functional, some sort of normalization is required. The most geometrically natural and standard normalization is done by restricting the $k$-th eigenvalue functional (\ref{EigenFunctional}) to conformal metrics with fixed volume. The study of this functional under such volume constraint is equivalent to study
\begin{equation}
\tilde g \in [g] \longmapsto \lambda_k(L_{\tilde g})\text{Vol}(M^n,\tilde g)^{\frac{2}{n}}.
\end{equation}
If one writes $\tilde g = \mu^{4/(n-2)}g$, this translates to
\begin{equation}\label{VolumeNorm}
\tilde g \in [g] \longmapsto \lambda_k(L_{\tilde g}) \left(\int_M \mu^{\frac{2n}{n-2}}\;dv_g\right)^{\frac{2}{n}}
\end{equation}

%We would like to highlight some important results concerning the existence of extremal metrics (see Definition \ref{Extremal}) for (\ref{VolumeNorm}). For $k=1$, if $Y(M^n,[g])>0$, then the infimum of $\lambda_1(L_{\tilde g})\text{Vol}(M^n,\tilde g)^{2/n}$ over $\tilde g\in[g]$ coincides with $Y(M^n,[g])$. If $Y(M^n,[g])<0$, then it is the supremum that coincides with $Y(M^n,[g])$. Therefore, proving existence of extremal metrics, either maximal or minimal, for the normalized eigenvalue functional (\ref{VolumeNorm}) when $k=1$ is equivalent to solving the Yamabe Problem. The existence theory for $k>1$ is a bit more delicate. In the case of $Y(M^n,[g])>0$, Ammann-Humbert proved that there is no Riemannian metric achieving the infimum of $\lambda_2(L_{\tilde g})\text{Vol}(M^n,\tilde g)^{2/n}$ over $\tilde g\in[g]$ (\cite{Ammann}). However, they managed to prove existence in the space of generalized conformal metrics (see Section 2 in \cite{Ammann} for definition), and, as an application, they obtained nodal solutions for a Yamabe type equation. When $Y(M^n,[g])<0$ and the number of negative eigenvalues is at least $2$, Gursky and the author proved the existence of (possibly generalized) conformal metrics maximizing (\ref{VolumeNorm}) when $k=2$ (\cite{Gursky}). Interestingly, we showed that any such metric defines either nodal solutions of a Yamabe type equation, or harmonic maps in a sphere. 

For the positive part of the spectrum of $L_g$, Ammann-Jammes proved that the volume normalization is not suitable if one is looking for Riemannian metrics achieving the supremum of the $k$-th eigenvalue functional in conformal classes (\cite{Ammann2}). Specifically, if $\lambda_+(L_g)$ is the first positive eigenvalue for $L_g$, then
\begin{equation}
\Lambda_+(M^n,[g]):=\sup_{\tilde g \in [g]} \lambda_{+}(L_{\tilde g}) \text{Vol}(M,\tilde g)^{\frac{2}{n}} = \infty.
\end{equation}
For instance, if $n\ge 3$, then $\sup_{\tilde g \in [g_r]} \lambda_{1}(L_{\tilde g}) \text{Vol}(\mathbb{S}^n,\tilde g)^{2/n} = \infty$, where $g_r$ is the round metric. This is part of a much more general phenomenon: if $A_g$ is a conformally covariant elliptic operator whose order is less than the dimension of the manifold, and if $A_g$ is invertible on the cylinder $\mathbb{S}^{n-1}\times \mathbb{R}$, then 
\begin{equation}
\bar \Lambda_+(M^n,[g]):=\sup_{\tilde g \in [g]} \lambda_{+}(A_{\tilde g}) \text{Vol}(M,\tilde g)^{\frac{2}{n}} = \infty,
\end{equation}
where $\lambda_+(A_g)$ denotes the first positive eigenvalue of $A_g$ (\cite{Ammann2}). This puts a restriction on the type of operators for which one can try to find extremal eigenvalues by maximizing over conformal classes. 

The main idea by Sire-Xu in \cite{Sire} was to show that there is another way to normalize (\ref{EigenFunctional}) such that the supremum over conformal metrics is finite, even for the positive part of spectrum of $L_g$. Let us define $F_g^k$ by
\begin{equation}
\tilde g = \mu^{\frac{4}{n-2}}g \in [g] \longmapsto F^k_g(\tilde g) := \lambda_k(L_{\tilde g}) \int_M \mu^{\frac{4}{n-2}}\;dv_g.
\end{equation}
Sire-Xu showed that for any metric $g$ on $M^n$, the quantity
\begin{equation}\label{SX-Norm}
\Lambda_k(M^n,g):=\sup_{\tilde g = \mu^{\frac{4}{n-2}}g \in [g]}\lambda_k(L_{\tilde g})\int_M \mu^{\frac{4}{n-2}}\;dv_g,
\end{equation}
 is finite (see Section 6 in \cite{Sire}). Unlike previously considered quantities, $\Lambda_k(M^n,g)$ is not conformally invariant, i.e. it depends on the choice of the metric $g$. However, questions regarding  the existence of maximal metrics for (\ref{SX-Norm}), their regularity, and its geometric meaning could still be asked. 

%\begin{theorem}[Sire-Xu (\cite{Sire})]
%Let $(M^n,g=\phi^*g_r)$ ($n\ge 3$) be isometrically embedded in $(S^N,g_r)$. If $\tilde g \in [g]$, then 
%\begin{equation}
%\lambda_k(L_{\tilde g}) \int_M \mu^{\frac{4}{n-2}}\;dv_g \le C(n)\left(V_c(N,\phi)k^{\frac{2}{n}} + \sup_M R_g\right),
%\end{equation}
%where $V_c(N,\phi)$ is the N-conformal volume of $\phi$, and $C(n)$ is a dimensional constant. 
%\end{theorem}

 %Note that if $\text{Vol}(M^n,g) = 1$, then the integral $\int_M\mu^{4/(n-2)}\;dv_g$ is smaller than the second factor appearing in (\ref{VolumeNorm}) by Hölder's inequality.

 In the case of the $k$-th eigenvalue functional under volume normalization,  (\ref{VolumeNorm}), harmonic maps have been found to be in connection with maximal metrics. Let $[g]$ be a conformal class for which $L_g$ has at least two negative eigenvalues and trivial kernel. Gursky and the author proved in \cite{Gursky} that to each generalized metric $\tilde g=\mu^{4/(n-2)}g$ (see Section 2 in \cite{Gursky} for definition) maximizing $(\ref{VolumeNorm})$ with $k=2$, there is a collection of $C^{2,\alpha}$-functions  $\{u_1,\cdots,u_p\}$, each of them solving $L_g(u_i)=\lambda_2(L_{\tilde g})u_i\mu^{4/(n-2)}$, such that $\sum_{i=1}^pu_i^2=\mu^2$. If $p>1$, then $(M^n\setminus \{\mu=0\},\tilde g)$ admits a harmonic map into a sphere. Our first result is a step in the same direction:

\begin{theorem}\label{map-SX}
Let $(M^n,g)$ be a closed Riemannian manifold of dimension $n\ge 3$ and with a unit volume metric $g$. Let $g_e=\mu_e^{4/(n-2)}g$ be a smooth extremal metric for $F_g^k$ (see Definition \ref{Extremal}) with $\int_M\mu_e^{4/(n-2)}\;dv_g=1$. If  either  
\begin{equation}\label{tech}
\lambda_k(L_{g_e}) > \lambda_{k-1}(L_{g_e})\text{\;\; or \;\;}\lambda_k(L_{g_e}) < \lambda_{k+1}(L_{g_e}),
\end{equation}
and $\lambda_k(L_{g_e})\not = 0$, then there exists a finite family $\{u_1,\cdots,u_p\}$ of eigenfunctions associated to $\lambda_k(L_{g_e})$  such that $\sum_{i=1}^pu_i^2\mu_e^2\equiv 1$. In particular, the map $U=(u_1\mu_e,\cdots,u_p\mu_e): M \rightarrow \mathbb{S}^{p-1}$ is well defined. As a consequence, 
\begin{equation}\label{ForEigen}
\lambda_k(L_{g_e}) = -\frac{1}{2}\mu_e^2\Delta_{g_e}(\mu_e^{-2}) + \mu_e^2\sum_{i=1}^p|\nabla_{g_e}u_i|_{g_e}^2 + c_nR_{g_e}.
\end{equation}
Furthermore, if the extremal metric $g_e$ coincides with the background metric $g$, i.e. if $\mu_e=1$, then $(M^n,g_e)$ admits a harmonic map into a sphere.
\end{theorem}

Notice that for any metric $\tilde g\in[g]$, the first eigenvalue $\lambda_1(L_{\tilde g})$ always satisfies the second condition in (\ref{tech}). As a consequence of Theorem \ref{map-SX}, we derive a necessary condition for the existence of extremal metrics for $F_g^1$: the sign of $R_g$ and $Y(M^n,[g])$ need to coincide. Corollary \ref{nec1-SX-P} discusses the result in the case when $Y(M^n,[g])>0$, but the same holds in the case of negative Yamabe invariant; see Corollary \ref{nec1-SX-N}.

\begin{corollary}\label{nec1-SX-P}
Suppose $[g]$ is a conformal class for which $Y(M^n,[g])>0$, and assume that $g_e=\mu_e^{4/(n-2)}g$ is an extremal metric for $F^1_g$ with $\int_M\mu_e^{4/(n-2)}\;dv_g = 1$. Then 
\begin{equation}\label{nec1-SX-P1}
 c_nR_g = \lambda_1(L_{g_e})\mu_e^{\frac{4}{n-2}}.
\end{equation}
In particular, the scalar curvature $R_g$ with respect to the background metric $g$ is positive everywhere.
\end{corollary}

We also solve the existence problem of maximal metrics for the first eigenvalue functional under Sire-Xu normalization, i.e. we solve the maximization problem for $F_g^1$. In fact, the condition of both $R_g$ and $Y(M^n,[g])$ having the same sign is also sufficient. Theorem \ref{Existence1-SX-P} discusses the case when $Y(M^n,[g])>0$, but the same conclusions hold in the case of negative Yamabe invariant; see Theorem \ref{Existence1-SX-N}.

\begin{theorem}\label{Existence1-SX-P}
Let M be a closed n-dimensional manifold endowed with a conformal class [g] satisfying $Y(M,[g])>0$. If $R_g>0$, then the metric 
\begin{equation}g_{\text{\tiny max}} := \mu_{\text{\tiny max}}^{\frac{4}{n-2}}g = \left(\frac{R_g}{\int_MR_gdv_g}\right)g
\end{equation}
solves the maximization problem for $\Lambda_1(M,g)$, i.e. $F_g^1(g_{\text{\tiny max}}) = \Lambda_1(M,g)$. 
\end{theorem}

Theorems \ref{Existence1-SX-P} and \ref{Existence1-SX-N} provide a generalization of the results found in \cite{Sire}, where the authors showed that the round metric was the unique maximizer for $F^1_{g_r}$ on $\mathbb{S}^n$. With these results the existence problem for $k=1$ is fully understood. The existence theory for higher eigenvalues is still unknown.

The paper is organized as follows. In Section \ref{Perturbation} we recall an important result from classical analytic perturbation theory. In particular, we discuss a result due to Canzani (\cite{Canzani}) which is the base for this work. The techniques employed to prove Theorem \ref{OSD} follow closely those used by Soufi-Ilias in \cite{ElSoufi}. In Section \ref{ExtremalMetrics} is where we proved  our main results, including those analogous to Corollary \ref{nec1-SX-P} and Theorem \ref{Existence1-SX-P} in the case when $Y(M^n,[g])<0$. The proof of Theorem \ref{map-SX} uses classical separation theorems as it was done in \cite{ElSoufi}, \cite{ElSoufi2}.

%%%%%%%%%%%%%%%%%%%%%%%%%%%%%%%
\section{Background On Analytic Perturbation Theory}\label{Perturbation}
%%%%%%%%%%%%%%%%%%%%%%%%%%%%%%%

We will introduce a more general notion than that of maximal metrics. Loosely speaking, extremal metrics will be those which are ``critical points" of the functional $\lambda_k(L_{.})$ under the Sire-Xu normalization. As we shall discuss, maximal metrics are always extremal, but the converse may not be true. To formally define what extremal metrics are, we need first to show the existence of the one-sided derivatives of $\lambda_k(L_{.})$ along conformal analytic perturbations. 

To this end, let $g(t)=\mu_t^{4/(n-2)}\tilde g$ be any analytic deformation of $\tilde g \in [g]$, i.e. $\mu_0 \equiv 1$ and $\mu_t$ depends real analytically in $t$. It is good to note that we are only considering deformations that stay in a given conformal class. The function $\mu_t$ is called the \textit{generating function} for the deformation $g(t)$. In general, the eigenvalue functional $\lambda_k(L_{g(t)})$ is continuous, but not differentiable. However, and as we mentioned before, both $\frac{d}{dt}\lambda_k(L_{g(t)})|_{t=0^+}$ and $\frac{d}{dt}\lambda_k(L_{g(t)})|_{t=0^-}$ exist. The existence of the one-sided derivatives relies on the following theorem from perturbation theory for linear operators. The original theory traces back to Rellich-Kato's work, but it was Canzani in \cite{Canzani} who proved that such a theory could be applied to a certain class of conformally covariant operators acting on smooth bundles over $M^n$; see also \cite{Bando}. 
 
 %%%%%%%%%%%%%%%%%%%%%%%%%%%%%%%%%%%%%%%
\begin{theorem}[Rellich-Kato, Canzani]\label{RKC}
Let $\lambda_k(L_{\tilde g})$ be the $k$-th eigenvalue of the Conformal Laplacian with respect to $\tilde g\in[g]$, and denote by $m$ its multiplicity. Pick any analytic deformation $g(t)=\mu_t^{4/(n-2)}\tilde g$ of $\tilde g$. Then there exist $\Lambda_1(t),\cdots,\Lambda_m(t)$ analytic in $t$, and $u_1(t),\cdots,u_m(t)$ convergent power series in $t$ with respect to the $L^2$ norm topology, such that
\begin{equation}\label{RKC1}
L_{g(t)}u_i(t) = \Lambda_i(t)u_i(t), \text{ with } \Lambda_i(0)=\lambda_k(L_{\tilde g}) \text{ for all } i=1,\cdots,m;
\end{equation}
and
\begin{equation}\label{RKC2}
\int_M u_i(t)u_j(t)\;dv_{g(t)} = \delta_{ij} \text{ for all } i=1,\cdots,m.
\end{equation}
Moreover, if we select positive constants $d_1$ and $d_2$ such that the spectrum of $L_{\tilde g}$ in $(\lambda_k(L_{\tilde g}) -d_1,\lambda_k(L_{\tilde g}) + d_2)$ consists only of $\lambda_k(L_{\tilde g})$, then one can find a small enough $\delta >0$ such that the spectrum of $L_{g(t)}$ in the same interval $(\lambda_k(L_{\tilde g}) -d_1,\lambda_k(L_{\tilde g}) + d_2)$ consists of $\{\Lambda_1(t),\cdots,\Lambda_m(t)\}$ alone for $|t|<\delta$.
\end{theorem}
%%%%%%%%%%%%%%%%%%%%%%%%%%%%%%%%%%%%%%%%

We find it important to notice what the theorem {\bf does not} assert. For a conformal metric $\tilde g\in[g]$, given an eigenvalue $\lambda_k(L_{\tilde g})$ with multiplicity $m$ and an orthonormal basis $\{u_i\}_{i=1}^m$ of the $k$-th eigenspace $E_k(L_{\tilde g})$, Theorem \ref{RKC} {\bf does not} provide a way to find a perturbation $g(t)=\mu_t^{4/(n-2)}\tilde g$ of $\tilde g$ with a collection of convergent power series $\{u_i(t)\}_{i=1}^m$ satisfying (\ref{RKC1}) and (\ref{RKC2}) such that $u_i(0)=u_i$. In fewer words, orthonormal bases of $E_k(L_{\tilde g})$ cannot be prescribed.

A crucial consequence of Theorem \ref{RKC} is:

%%%%%%%%%%%%%%%%%%%%%%%%%%%%%%%%%%%%%%%%%
\begin{theorem}\label{OSD}
 Let $g(t)=\mu_t^{4/(n-2)}\tilde g$ be any analytic deformation of $\tilde g$, i.e. $g(t)$ is analytic with respect to $t$ in a neighborhood of $t=0$, and $g(0)=\tilde g$ ($\mu_0\equiv1$). Then the one-sided derivatives of $\lambda_k(L_{g(t)})$,
\[
\frac{d}{dt}\lambda_k(L_{g(t)})|_{t=0^+}\text{ and }\frac{d}{dt}\lambda_k(L_{g(t)})|_{t=0^-},
\]
both exist. Moreover, an explicit formula can be computed in certain cases:
\begin{enumerate}[(i)]
\item If $\lambda_k(L_{\tilde g}) > \lambda_{k-1}(L_{\tilde g})$, then 
\begin{equation}\label{OSD+1}
\frac{d}{dt}\lambda_k(L_{g(t)})|_{t=0^+} = \min_{1\le i \le m} \Lambda'_i(0)  
\end{equation}
and
\begin{equation}\label{OSD-1}
\frac{d}{dt}\lambda_k(L_{g(t)})|_{t=0^-} = \max_{1\le i \le m} \Lambda'_i(0).
\end{equation}
\item If $\lambda_k(L_{\tilde g})<\lambda_{k+1}(L_{\tilde g})$, then
\begin{equation}\label{OSD+2}
\frac{d}{dt}\lambda_k(L_{g(t)})|_{t=0^+} = \max_{1\le i \le m} \Lambda'_i(0)
\end{equation}
and
\begin{equation}\label{OSD-2}
\frac{d}{dt}\lambda_k(L_{g(t)})|_{t=0^-} = \min_{1\le i \le m} \Lambda'_i(0).
\end{equation}
\end{enumerate}
In both cases,
\begin{equation}\label{AnalyticEigen}
\Lambda_i'(0) = -\frac{4}{n-2}\lambda_k(L_{\tilde g})\int_M hu_i^2\;dv_{\tilde g}.
\end{equation}
Here the notation is as in Theorem \ref{RKC}, and $h$ denotes $\frac{d}{dt}\mu_t|_{t=0}$.
\end{theorem}
%%%%%%%%%%%%%%%%%%%%%%%%%%%%%%%%%%%%%

Notice that if we know the sign of $\lambda_k(L_g)$, then the one-sided derivatives can be rewritten in terms of formula (\ref{AnalyticEigen}). For instance, if $\lambda_k(L_{\tilde g})>0$ and $\lambda_k(L_{\tilde g}) > \lambda_{k-1}(L_{\tilde g})$, so that we are in case (i), then
\begin{equation}
\frac{d}{dt}\lambda_k(L_{g(t)})|_{t=0^+} = -\frac{4}{n-2} \lambda_k(L_{\tilde g})\max_{1\le i \le m} \int_M hu_i^2\;dv_{\tilde g},
\end{equation}
and
\begin{equation}
\frac{d}{dt}\lambda_k(L_{g(t)})|_{t=0^-} = -\frac{4}{n-2} \lambda_k(L_{\tilde g}) \min_{1\le i\le m} \int_M hu_i^2\;dv_{\tilde g}.
\end{equation}

%%%%%%%%%%%%%%%%%%%%%%%%%%%%%%%%%
\begin{proof}
Let us start by pointing out that the family $\{\Lambda_1(t),\cdots,\Lambda_m(t)\}$ is an unordered family of eigenvalues of $L_{g(t)}$. From Theorem \ref{RKC}, in some neighborhood of $t=0$, the spectrum of $L_{g(t)}$ consists only of $\{\Lambda_1(t),\cdots,\Lambda_m(t)\}$. The continuity of $t\mapsto \lambda_k(L_{g(t)})$ and the analyticity of $t\mapsto \Lambda_i(t)$ in a neighborhood of $t=0$ implies that there exists $a,b\in\{1,\cdots,m\}$ such that
\begin{equation} \label{OSD1}
\lambda_k(L_{g(t)}) = \begin{cases}
			      \Lambda_a(t) & \hspace{.2in} \text{if } t\in[0,\eta) \\
			      \Lambda_b(t) & \hspace{.2in} \text{if } t\in(-\eta,0].
			      \end{cases},
\end{equation}
for some $\eta>0$. This shows the existence of the one-sided derivatives of $\lambda_k(L_{g(t)})$ at $t=0$.

Assume now that $\lambda_k(L_{\tilde g}) > \lambda_{k-1}(L_{\tilde g})$. By choosing $d>0$ such that $\lambda_{k-1}(L_{\tilde g})\not\in(\lambda_k(L_{\tilde g}) - d , \lambda_k(L_{\tilde g}) + d)$, thanks to Theorem \ref{RKC} one can select a small enough $\eta>\delta >0$ such that Spect($L_{g(t)}$)$\cap(\lambda_k(L_{\tilde g}) - d , \lambda_k(L_{\tilde g}) + d) = \{\Lambda_1(t),\cdots,\Lambda_m(t)\}$ for $|t|<\delta$. This together with the continuity of $\lambda_{k-1}(L_{g(t)})$ implies that $\Lambda_i(t)>\lambda_{k-1}(L_{g(t)})$. Therefore, $\Lambda_i(t)\ge\lambda_k(L_{g(t)})$ and thus $\lambda_k(L_{g(t)}) = \min_{1\le i\le m}\{\Lambda_i(t)\}$ for $|t|<\delta$. 

We now use $(\ref{OSD1})$ to deduce $\lambda_k(L_{g(t)}) = \Lambda_a(t)\le\Lambda_i(t)$ for all $i=1,\cdots,m$ and $t\in [0,\delta)$, and $\lambda_k(L_{g(t)}) = \Lambda_b(t)\le\Lambda_i(t)$ for all $i=1,\cdots,m$ and $t\in (-\delta,0]$. Using $\lambda_k(L_{g(t)}) = \Lambda_a(t)\le\Lambda_i(t)$ with $t$ and $i$ as specified, one obtains
\begin{equation}
\frac{\Lambda_a(t) - \lambda_k(L_{\tilde g})}{t} \le \frac{\Lambda_i(t) - \lambda_k(L_{\tilde g})}{t}\text{ for $t\in(0,\delta)$} \Rightarrow \Lambda_a'(0)  \le  \Lambda_i'(0).
\end{equation}
Similarly, from $\lambda_k(L_{g(t)}) = \Lambda_b(t)\le \Lambda_i(t)$, we deduce
\begin{equation}
\frac{\Lambda_b(t) - \lambda_k(L_{\tilde g})}{t} \ge \frac{\Lambda_i(t) - \lambda_k(L_{\tilde g})}{t}\text{ for $t\in(-\delta,0)$} \Rightarrow \Lambda_b'(0)  \ge  \Lambda_i'(0).
\end{equation}
Putting all together we get
\begin{equation}
\frac{d}{dt}\lambda_k(L_{g(t)})|_{t=0^+} = \min_{1\le i \le m} \Lambda_i'(0) \text{ and }\frac{d}{dt}\lambda_k(L_{g(t)})|_{t=0^-} = \max_{1\le i \le m} \Lambda_i'(0),
\end{equation}
which is what we wanted to show. If $\lambda_k(L_{\tilde g}) < \lambda_{k+1}(L_{\tilde g})$, then a similar analysis leads to $\lambda_k(L_{g(t)}) = \max_{1\le i\le m}\{\Lambda_i(t)\}$. The proof for this case is similar and it is therefore omitted.

It remains to show (\ref{AnalyticEigen}). We use the eigenvalue equation (\ref{RKC1}), together with the transformation law (\ref{ConfTrans}):
\begin{equation}
L_{\tilde g} (\mu_tu_i(t)) = \Lambda_i(t)\mu_t^{\frac{n+2}{n-2}}u_i(t).
\end{equation}
Differentiating this equation with respect to $t$ and then setting $t=0$ gives us
\begin{equation}\label{OSD2}
L_{\tilde g}(hu_i) + L_{\tilde g}(u_i') = \Lambda_i'(0)u_i + \frac{n+2}{n-2} \lambda_k(L_{\tilde g})hu_i + \lambda_k(L_{\tilde g}) u_i',
\end{equation}
where $u_i=u_i(0)$ and $u_i'= \frac{d}{dt}u_i(t)|_{t=0}$. Also, by setting $t=0$ in (\ref{RKC1}) one obtains $L_{\tilde g} u_i = \lambda_k(L_{\tilde g})u_i$, and multiplying this by $u_i'$ gives us
\begin{equation}\label{OSD3}
u_i'L_{\tilde g}u_i = \lambda_k(L_{\tilde g}) u_i'u_i.
\end{equation} 
On the other hand, one can multiply (\ref{OSD2}) by $u_i$ to get
\begin{equation}\label{OSD4}
u_iL_{\tilde g}(h u_i) + u_iL_{\tilde g}u_i' = \Lambda_i'(0) u_i^2 + \frac{n+2}{n-2}\lambda_k(L_{\tilde g})hu_i^2 + \lambda_k(L_{\tilde g}) u_iu_i'. 
\end{equation}
After subtracting (\ref{OSD3}) from (\ref{OSD4}), and integrating with respect to $dv_{\tilde g}$ we obtain
\begin{equation}
\begin{split}
\int_Mu_iL_{\tilde g}(hu_i)\;dv_{\tilde g} + \underbrace{\int_M(u_iL_{\tilde g}u_i' - u_i'L_{\tilde g}u_i)\;dv_{\tilde g}}_{=0\text{ by self-adjointness of }L_{\tilde g}} &= \Lambda_i'(0)\underbrace{\int_M u_i^2\;dv_{\tilde g}}_{=1\text{ by }(1.2)} \\ &\hspace{.1in}+ \frac{n+2}{n-2}\lambda_k(L_{\tilde g})\int_Mhu_i^2\;dv_{\tilde g}\\
\int_M L_{\tilde g}(u_i)\cdot(hu_i)\;dv_{\tilde g} &= \Lambda_i'(0) + \frac{n+2}{n-2}\lambda_k(L_{\tilde g})\int_Mh u_i^2\;dv_{\tilde g} \\
\lambda_k(L_{\tilde g})\int_M hu_i^2\;dv_{\tilde g} &= \Lambda_i'(0) + \frac{n+2}{n-2}\lambda_k(L_{\tilde g})\int_Mh u_i^2\;dv_{\tilde g}.
\end{split}
\end{equation}
Hence
\begin{equation}
\Lambda_i'(0) = -\frac{4}{n-2}\lambda_k(L_{\tilde g})\int_M hu_i^2\;dv_{\tilde g}.
\end{equation}
This concludes the proof of $(i)$. The arguments for $(ii)$ are similar and therefore we omit them.
\end{proof}
%%%%%%%%%%%%%%%%%%%%%%%%%%%%%%%

%%%%%%%%%%%%%%%%%%%%%%%%%%%%%%%
\section{Extremal Metrics under Sire-Xu Normalization}\label{ExtremalMetrics}
%%%%%%%%%%%%%%%%%%%%%%%%%%%%%%%

We start with Corollary \ref{OSD-SX}, which is an immediate consequence of Theorem \ref{OSD}.

%%%%%%%%%%%%%%%%%%%%%%%%%%%%%%%%%%%%
\begin{corollary}\label{OSD-SX}
Let $g(t) = \mu_t^{4/(n-2)}\tilde g = \mu_t^{4/(n-2)}\tilde \mu^{4/(n-2)}g$ be any analytic deformation of $\tilde g\in[g]$. Then the one-sided derivatives of $F^k_g(g(t))$ both exist. Let us further assume that the conformal factor of the perturbed metric satisfies
\begin{equation}\label{OSD-SX1}
\int_M \tilde \mu^{\frac{4}{n-2}}\;dv_g =1
\end{equation}
Then an explicit formula for the one-sided derivatives can be computed in certain cases:
\begin{enumerate}[(i)]
\item if $\lambda_k(L_{\tilde g})> \lambda_{k-1}(L_{\tilde g})$, then
\begin{equation}\label{OSD-SX2}
 \frac{d}{dt}F^k_g(g(t))|_{t=0^+} = \frac{4}{n-2}\lambda_k(L_{\tilde g})\int_M h \tilde \mu^{\frac{4}{n-2}}\;dv_g + \min_{1\le i\le m}\Lambda_i'(0)
\end{equation}
and
\begin{equation}\label{OSD-SX3}
 \frac{d}{dt}F^k_g(g(t))|_{t=0^-} = \frac{4}{n-2}\lambda_k(L_{\tilde g})\int_M h\tilde \mu^{\frac{4}{n-2}}\;dv_g + \max_{1\le i\le m} \Lambda_i'(0)
\end{equation}
If we select as generating function $\mu_t = e^{tw\tilde\mu^2}$, where $w$ has zero mean value with respect to $\tilde g$, then
\begin{equation}\label{OSD-SX4}
 \frac{d}{dt}F^k_g(g(t))|_{t=0^+} = \min_{1\le i\le m} \Lambda_i'(0),
\end{equation}
and
\begin{equation}\label{OSD-SX5}
 \frac{d}{dt}F^k_g(g(t))|_{t=0^-} = \max_{1\le i\le m} \Lambda_i'(0).
\end{equation}

\item if $\lambda_k(L_{\tilde g})< \lambda_{k+1}(L_{\tilde g})$, then
\begin{equation}\label{OSD-SX6}
 \frac{d}{dt}F^k_g(g(t))|_{t=0^+} = \frac{4}{n-2}\lambda_k(L_{\tilde g})\int_M h \tilde \mu^{\frac{4}{n-2}}\;dv_g + \max_{1\le i\le m}\Lambda_i'(0),
\end{equation}
and
\begin{equation}\label{OSD-SX7}
 \frac{d}{dt}F^k_g(g(t))|_{t=0^-} = \frac{4}{n-2}\lambda_k(L_{\tilde g})\int_M h\tilde \mu^{\frac{4}{n-2}}\;dv_g + \min_{1\le i\le m} \Lambda_i'(0). 
\end{equation}
If we select as generating function $\mu_t = e^{tw\tilde\mu^2}$, where $w\in L^2(M,\tilde g)$ has zero mean value with respect to $\tilde g$, then
\begin{equation}\label{OSD-SX8}
 \frac{d}{dt}F^k_g(g(t))|_{t=0^+} = \max_{1\le i\le m} \Lambda_i'(0)
\end{equation}
and
\begin{equation}\label{OSD-SX9}
 \frac{d}{dt}F^k_g(g(t))|_{t=0^-} = \min_{1\le i\le m} \Lambda_i'(0)
\end{equation}
\end{enumerate}
Here the notation is as in Theorem \ref{OSD}. 
\end{corollary}
%%%%%%%%%%%%%%%%%%%%%%%%%%%%%%%%%%%%

%%%%%%%%%%%%%%%%%%%%%%%%%%%%%%%%%%%%%%%%%
\begin{proof}
The proof follows directly from Theorem \ref{OSD}. Using 
\begin{equation}
F_g^k(g(t)) = \lambda_k(L_{g(t)})\int_M \mu_t^{\frac{4}{n-2}}\cdot \mu^{\frac{4}{n-2}}\;dv_g
\end{equation}
 we see that the one-sided derivatives of $F^k_g(g(t))$ both exist. Assuming $\lambda_k(L_{\tilde g})> \lambda_{k-1}(L_{\tilde g})$, we compute $\frac{d}{dt}F^k_g(g(t))|_{t=0^+}$ as follows: 
 
\begin{equation}
\begin{split}
\frac{d}{dt}F^k_g(g(t))|_{t=0^+} &= \frac{d}{dt}\lambda_k(L_{g(t)})|_{t=0^+}\cdot \int_M\mu_0^{\frac{4}{n-2}}\tilde \mu^{\frac{4}{n-2}}\;dv_g\\ & \hspace{.1in}+ \frac{4}{n-2}\lambda_k(L_{\tilde g}) \int_M \mu_0^{\frac{6-n}{n-2}}\tilde \mu^{\frac{4}{n-2}}h\;dv_g\\ &=\min_{1\le i\le m} \Lambda_i'(0)\cdot \underbrace{\int_M \tilde \mu^{\frac{4}{n-2}}\;dv_g}_{=1\text{ by assumption on }\tilde \mu} + \frac{4}{n-2} \lambda_k(L_{\tilde g}) \int_M h\tilde \mu^{\frac{4}{n-2}}\;dv_g,
\end{split}
\end{equation}
which is the desired expression. The computations for $\frac{d}{dt}F^k_g(g(t))|_{t=0^-}$, as well as for the case $\lambda_k(L_{\tilde g})<\lambda_{k+1}(L_{\tilde g})$, are similar and are therefore omitted. This concludes the proof.
\end{proof}
%%%%%%%%%%%%%%%%%%%%%%%%%%%%%%%%%%%%%%%%%%%%%

The previous result allows us to finally define what an extremal metric is.
\begin{definition}\label{Extremal}
The metric $\tilde g\in[g]$ is said to be \textit{extremal} for the functional $F^k_g$ if for any deformation $g(t)$ of $\tilde g$ which is analytic in a neighborhood of $t=0$, we have 
\begin{equation}
\frac{d}{dt}F^k_g(g(t))|_{t=0^+}\cdot \frac{d}{dt}F^k_g(g(t))|_{t=0^-}\le 0.
\end{equation}
\end{definition}

These extremal metrics do not only catch global maximums, but also local extremums, including minimums. How to distinguish these different types of extremal metrics is another important question. Note that solving the maximization problem for $\Lambda_k(M,g)$ is equivalent to finding $g_{\text{\tiny max}}\in[g]$ such that $F_g^k(g_{\text{\tiny max}}) = \Lambda_k(M,g)$. In particular, any maximal metric is extremal.

In Section \ref{Introduction}, we discuss how in certain cases the existence of maximal metrics for (\ref{VolumeNorm}) is associated with existence of harmonic maps into spheres. The connection between extremal eigenvalues and harmonic maps was found first in the case of eigenvalues for the Laplace-Beltrami operator under volume normalization. Here, if $g_e\in[g]$ is an extremal metric (defined analogously) for $\tilde g \in [g] \mapsto \lambda_k(-\Delta_{\tilde g}) \text{Vol}(M^n,\tilde g)^{2/n}$, then the Riemannian manifold $(M^n,g_e)$ admits a harmonic map with constant energy in a sphere (\cite{ElSoufi}). Existence of maximal metrics for the aforementioned functional, which are a particular case of extremal metrics, have only been proved in two dimensions (\cite{Nadirashvili, Nadirashvili2, Petrides, Petrides2}). We would like to point out that a map $U: (M^n,g) \rightarrow (\mathbb{S}^{p-1},g_r)$ is harmonic if $U=(u_1,\cdots,u_p)$ and each coordinate function is an eigenfunction associated to the same eigenvalue $\lambda_k(-\Delta_g)$. 

\begin{proof}[Proof of Theorem \ref{map-SX}]
Let us denote by $E_k(L_{g_e})$ the eigenspace corresponding to the $k$-th eigenvalue $\lambda_k(L_{g_e})$, and consider the subset $K\subset L^2(M^n,g_e)$ defined by 
\begin{equation}
K:=\{u^2\mu_e^2:\; u\in E_k(L_{g_e})\text{ and }\|u\|^2_{L^2(M,g_e)}=1\}.
\end{equation}
This subset is compact and lies in a finite dimensional subset of $L^2(M,g_e)$. By Caratheodory's Theorem for Convex Hulls one deduce that the convex hull of $K$,
\begin{equation}
\text{Conv}(K) = \left\{\sum_{\text{finite}} a_ju_j^2\mu_e^2: a_j\ge 0, \sum a_j = 1, u_j\in E_k(L_{g_e}),\|u_j\|^2_{L^2(M,g_e)}=1\right\},
\end{equation}
is compact as well. 

The first step of the proof is to show that $1\in \text{Conv}(K)$. On the contrary, let us assume that $\{1\}$ and $\text{Conv}(K)$ have empty intersection. Since they are both convex and nonempty, the former one closed and the latter one compact, Hahn Banach Separation Theorem gives us the existence of a functional $\Phi \in (L^2(M,g_e))^*$ separating $\{1\}$ from $\text{Conv}(K)$:  
\begin{equation}\label{map1}
\Phi(1) >0
\end{equation}
and
\begin{equation}\label{map2}
\Phi(\varphi) \le 0,\;\;\forall \varphi\in\text{Conv}(K).
\end{equation}
Furthermore, Riesz's Representation Theorem provides us with the existence of a function $f\in L^2(M,g_e)$ such that $\Phi$ is defined by integrating against $f\;dv_{g_e}$. This allow us to rewrite $(\ref{map1})$ and $(\ref{map2})$ as 
\begin{equation}\label{map3}
\int_M f\;dv_{g_e} >0
\end{equation}
and
\begin{equation}\label{map4}
\int_M\varphi f\;dv_{g_e} \le 0,\;\;\forall \varphi\in\text{Conv}(K).
\end{equation}

We let $w = f - \int_M f\;dv_{g_e}$ be the zero mean value part of $f$, and consider the analytic deformation $g(t)$ of $g_e$ generated by the function $\mu_t = e^{tw\mu_e^2}$. As before, we use $h$ to denote $\frac{d}{dt}\mu_t|_{t=0} = w\mu_e^2$.  From $(\ref{map3})$ and $(\ref{map4})$ one gets that for all $u\in E_k(L_{g_e})$,
\begin{equation}
\int_M hu^2\;dv_{g_e} =\underbrace{\int_M u^2\mu_e^2f\;dv_{g_e}}_{\le 0} - \underbrace{\left(\int_Mf\;dv_{g_e}\right)}_{>0}\underbrace{\left(\int_M u^2\mu_e^2\;dv_{g_e}\right)}_{>0}.
\end{equation}\label{map5}
This implies that the quadratic form
\begin{equation}
-\frac{4}{n-2}\lambda_k(L_{g_e})\int_M hu^2\;dv_{g_e}
\end{equation}
has a constant sign on $E_k(L_{g_e})$. 

Let $m$ be the multiplicity of $\lambda_k(L_{g_e})$, and denote by $\{u_i\}_{i=1}^m$ the orthonormal basis arising from the generating function $\mu_t$ as in Theorems \ref{RKC} and \ref{OSD}. Then
\begin{equation}\label{map6}
\Lambda_i'(0) = - \frac{4}{n-2}\lambda_k(L_{g_e})\int_M hu_i^2\;dv_{g_e}
\end{equation}
has constant sign for all $i=1,\cdots,m$. By means of Corollary  \ref{OSD-SX} we deduce that the one-sided derivatives of $F_g^k(\mu_t^{4/(n-2)}g_e)$ at $t=0$ are both of the same sign. This contradicts the extremality assumption on $g_e$. Hence, $1\in\text{Conv}(K)$, i.e. there exists a finite collection $u_1,\cdots, u_p\in E_k(L_{g_e})$ such that $\sum_{i=1}^p u_i^2\mu_e^2 \equiv 1$. 

We now show that the equation $(\ref{ForEigen})$ holds. To this end, we apply the Laplace-Beltrami operator to the relation $\sum_{i=1}^pu_i^2\equiv \mu_e^{-2}$:

\begin{equation}
\begin{split}
-\frac{1}{2}\Delta_{g_e}(\mu_e^{-2})& =-\frac{1}{2}\Delta_{g_e}\left(\sum_{i=1}^p u_i^2\right) = -\sum_{i=1}^p\left(u_i\Delta_{g_e}u_i + |\nabla_{g_e}u_i|_{g_e}^2\right) \\ & = \sum_{i=1}^p\left(u_iL_{g_e}u_i  - |\nabla_{g_e}u_i|_{g_e}^2 - c_nR_{g_e}u_i^2\right) \\ & = \lambda_k(L_{g_e})\mu_e^{-2} - \sum_{i=1}^p |\nabla_{g_e}u_i|_{g_e}^2 - c_nR_{g_e}\mu_e^{-2}.
\end{split}
\end{equation}
The result is obtained after solving for $\lambda_k(L_{g_e})$.

Finally, if the extremal metric $g_e$ coincides with the background metric $g$, then $\mu_e=1$. Therefore, from (\ref{ForEigen}),
\begin{equation}\label{ForEigen2}
\lambda_k(L_g) = c_nR_g + \sum_{i=1}^p|\nabla_g u_i|^2.
\end{equation}
This gives us
\begin{equation}
\begin{split}
L_gu_j = \lambda_k(L_g) u_j & \iff (-\Delta_g +c_nR_g)u_j = \left(c_nR_g + \sum_{i=1}^p|\nabla_g u_i|^2\right)u_j \\ &\iff -\Delta_gu_j = \left(\sum_{i=1}^p|\nabla_gu_i|^2\right)u_j,
\end{split} 
\end{equation}
which is the harmonic map equation for maps into spheres (\cite{Helein}). This finishes the proof.
\end{proof}
%%%%%%%%%%%%%%%%%%%%%%%%%%%%%%%%%%%%%%%%
We remark that equation (\ref{ForEigen2}) implies
\begin{equation}
\lambda_k(L_g)\ge c_nR_g.
\end{equation}
Therefore, if $\lambda_k(L_g)<0$, then $R_g<0$ everywhere on $M^n$.

%%%%%%%%%%%%%%%%%%%%%%%%%%%%%%%%%%%%%%
\subsection{${\bf \Lambda_1(M,g)}$ for the case ${\bf Y(M,[g])>0}$}
%%%%%%%%%%%%%%%%%%%%%%%%%%%%%%%%%%%%%%

We are given a closed n-dimensional Riemannian manifold $M$ endowed with a conformal class $[g]$ satisfying $Y(M,[g])>0$. We focus now on $\lambda_1(L_{\tilde g})$, where $\tilde g \in [g]$. It is well known that the multiplicity of $\lambda_1(L_{\tilde g})$ is always $1$. 

Proposition \ref{nec1-SX-P} implies that the maximization problem ``finding $\tilde g \in [g]$ such that $F^1_g(\tilde g) = \Lambda_1(M,g)$" would not have a solution if the scalar curvature $R_g$ with respect to the reference metric $g$ does not satisfies $R_g > 0$ everywhere on $M^n$. That is, $R_g>0$ is necessary. Finding such a reference metric within a conformal class is always possible as long as $Y(M,[g])>0$, which we are assuming.

%%%%%%%%%%%%%%%%%%%%%%%%%%%%%

\begin{proof}
Since $\lambda_1(L_{g_e})$ has multiplicity one, it satisfies $\lambda_1(L_{g_e})< \lambda_2(L_{g_e})$, and so Theorem \ref{map-SX} gives us that $u^2 = \mu_e^{-2}$, where $u$ is an eigenfunction associated to $\lambda_1(L_{g_e})$. Therefore $u=\pm \mu_e^{-1}$, and 
\begin{equation}
\lambda_1(L_{g_e})(\pm\mu_e^{-1}) = \lambda_1(L_{g_e})u = L_{g_e}(u) = L_{g_e}(\pm\mu_e^{-1}).
\end{equation}
Using the transformation law (\ref{ConfTrans}), we get
\begin{equation}
\lambda_1(L_{g_e})(\mu_e^{-1}) = \mu_e^{-\frac{n+2}{n-2}}L_{g}(\mu_e\mu_e^{-1}) =\mu_e^{-\frac{n+2}{n-2}}L_{g}(1) = \mu_e^{-\frac{n+2}{n-2}}(c_nR_g)
\end{equation}
The result is obtained after solving for $c_nR_g$.
\end{proof}

We remark that Proposition \ref{nec1-SX-P} follows directly from Corollary \ref{OSD-SX} and the simplicity of $\lambda_1(L_{g_e})$. Indeed, since $E_1(L_{g_e})$ is generated by one function, say $u$, formulas $(\ref{OSD-SX6})$ and $(\ref{OSD-SX7})$ are equal to each other. That is, the functional $F_g^1$ is differentiable at $t=0$ along any analytic deformation $g(t)$ of $g_e$. Then the extremal condition on $g_e$ implies that for any $h\in C^\infty(M)$,
\begin{equation}
\frac{4}{n-2}\lambda_1(L_{g_e})\left\{\int_Mh\mu_e^{\frac{4}{n-2}}\;dv_g - \int_M hu^2dv_{g_e}\;\right\} = 0.
\end{equation} 
This is equivalent to 
\begin{equation}
\int_Mh\left(\mu_e^{-2} - u^2\right)\;dv_{g_e} = 0.
\end{equation}
Since $h$ is arbitrary, we deduce that $\mu_e^{-2}=u^2$, and the result follows in a similar manner. 

%A similar argument shows that, beside the case $k=1$, if a metric $g_e=\mu^{4/(n-2)}g$ is extremal for $F_g^k$ and $\int_M\mu^{4/(n-2)}\;dv_{g_e} = 1$, then it cannot be simple. This is due to the fact that, if otherwise, then one would have a constant sign eigenfunction for $L_{g_e}$, which is only possible for $k=1$. We summarize this discussion in the following proposition.

%%%%%%%%%%%%%%%%%%%%%%%%%%%%%%%%%%%%%%%
%\begin{proposition}\label{Degenerate}
%Let $k>1$. Assume that for $g_e\in[g]$ we either have $\lambda_k(L_{g_e}) > \lambda_{k-1}(L_{g_e})$ or $\lambda_k(L_{g_e}) < \lambda_{k+1}(L_{g_e})$, and that $\lambda_k(L_{g_e})\not = 0$. If $g_e=\mu_e^{4/(n-2)}g$ is an extremal metric for $F^k_g$ with $\int_M \mu_e^{4/(n-2)}\;dv_g=1$, then $\lambda_k(L_{g_e})$ is degenerate.
%\end{proposition} 
%%%%%%%%%%%%%%%%%%%%%%%%%%%%%%%%%%%%%%%

In terms of maximal metrics, the condition $R_g>0$ everywhere on $M^n$ is also sufficient for existence of conformal metrics achieving $\Lambda_1(M^n,g)$. This is the result stated in Theorem \ref{Existence1-SX-P}.

\begin{proof}[Proof of Theorem \ref{Existence1-SX-P}]
Let us start by noticing that $\int_M\mu_{\text{\tiny max}}^{4/(n-2)}\;dv_g = 1$. We first show that $c_n\int_MR_g\;dv_g$ is an eigenvalue for $L_{g_{\text{\tiny max}}}$ with eigenfunction $\mu_{\text{\tiny max}}^{-1}$, and, moreover, $\lambda_1(L_{g_{\text{\tiny max}}}) = c_n\int_MR_g\;dv_g$. This is a consequence of the conformal properties of the conformal laplacian:
\begin{equation}
\begin{split}
L_{g_{\text{\tiny max}}}(\mu_{\text{\tiny max}}^{-1}) & = \mu_{\text{\tiny max}}^{-\frac{n+2}{n-2}}L_g(1) = \mu_{\text{\tiny max}}^{-\frac{n+2}{n-2}}c_nR_g\\& = \mu_{\text{\tiny max}}^{-\frac{n+2}{n-2}}c_n\cdot\mu_{\text{\tiny max}}^{\frac{4}{n-2}} \int_MR_g\;dv_g = \left(c_n\int_MR_g\;dv_g\right)\mu_{\text{\tiny max}}^{-1}.
\end{split}
\end{equation}
To argue that $c_n\int_MR_g\;dv_g$ is, in fact, the smallest eigenvalue of $L_{g_{\text{\tiny max}}}$, i.e. $\lambda_1(L_{g_{\text{\tiny max}}})= c_n\int_MR_g\;dv_g$, we observe that the eigenfunction $\mu_{\text{\tiny max}}^{-1}$ has constant sign.

It remains to show that $F^1_g(g_{\text{\tiny max}}) = \Lambda_1(M,g)$. To this end, pick an arbitrary metric $\tilde g = \tilde \mu^{4/(n-2)}g\in[g]$. By the variational characterization of $\lambda_1({L_{\tilde g}})$ with $u=\tilde \mu^{-1}$ as test function, we get
\begin{equation}
\begin{split}
\lambda_1(L_{\tilde g})\cdot \int_M\tilde\mu^{\frac{4}{n-2}}\;dv_g &\le \frac{\displaystyle\int_M \tilde \mu^{-1}L_{\tilde g}(\tilde \mu^{-1})\;dv_{\tilde g} }{\displaystyle \int_M \tilde \mu^{-2}\;dv_{\tilde g}}\cdot \int_M\tilde\mu^{\frac{4}{n-2}}\;dv_g\\&= \frac{\displaystyle\int_M \tilde \mu^{-1}\cdot \tilde \mu\; L_{g}(\tilde\mu\cdot\tilde \mu^{-1})\;dv_{g} }{\displaystyle \int_M \tilde \mu^{-2}\tilde\mu^{\frac{2n}{n-2}}\;dv_{g}}\cdot \int_M\tilde\mu^{\frac{4}{n-2}}\;dv_g\\&=\frac{\displaystyle\int_M L_g(1)\;dv_{g} }{\displaystyle \int_M \tilde\mu^{\frac{4}{n-2}}\;dv_{g}}\cdot \int_M\tilde\mu^{\frac{4}{n-2}}\;dv_g = c_n\int_M R_g\;dv_g = \lambda_1(L_{g_{\text{\tiny max}}}).
\end{split}
\end{equation}
This finishes the proof.
\end{proof}

As a particular case of the previous theorem, we get Sire-Xu Theorem on the existence of maximal metrics for $\Lambda_1(\mathbb{S}^n,g_r)$, where $g_r$ is the standard round metric on the sphere $\mathbb{S}^n$.

%%%%%%%%%%%%%%%%%%%%%%%
\begin{theorem}[Sire-Xu]{\label{Sire-Xu}}
For the standard round sphere $(\mathbb{S}^n,g_r)$, the only maximizer for $\Lambda_1(\mathbb{S}^n,g_r)$, up to scalings, is the standard round metric $g_r$.
\end{theorem}
%%%%%%%%%%%%%%%%%%%%%%%

%%%%%%%%%%%%%%%%%%%%%%%
\begin{proof}
It is well known that $R_{g_r}=n(n-1)$. From Theorem \ref{Existence1-SX-P} we get that $g_{\text{\tiny max}}= \omega_n^{-1}g_r$ solves the maximization problem for $F^1_{g_r}$. Here $\omega_n$ denotes the volume of $\mathbb{S}^n$ with respect to $g_r$. If $\tilde g = \tilde\mu^{4/(n-2)}g_r\in[g_r]$ maximizes $F^1_{g_r}$ , i.e. if $F^1_{g_r}(\tilde g) = \Lambda_1(\mathbb{S}^n,g_r)$, then it has to satisfies equation $(\ref{nec1-SX-P1})$. This implies that the conformal factor $\tilde g$ is a constant function. The proof is now completed.
\end{proof}
%%%%%%%%%%%%%%%%%%%%%%%

We end this section with a comment on the uniqueness of maximal metrics for $\lambda_1(L_{.})$. If we have a maximal metric for $F_g^1$ and the scalar curvature of the reference metric $g$ is constant, then the conformal factor of the maximal metric is constant due to (\ref{nec1-SX-P1}). This is the scenario in Theorem \ref{Sire-Xu}.

%%%%%%%%%%%%%%%%%%%%%%%%%%%%%%%%%%%%%
\subsection{${\bf \Lambda_1(M,g)}$ for the case ${\bf Y(M^n,[g])<0}$}
%%%%%%%%%%%%%%%%%%%%%%%%%%%%%%%%%%%%%

In this case, $\lambda_1(L_{\tilde g})$ is strictly negative for any $\tilde g\in [g]$, and its multiplicity is still 1. Formula (\ref{nec1-SX-P1}) in Proposition \ref{nec1-SX-P} remains valid in this case. For the sake of completeness, we state the result but  omit the proof as it is the same argument.

%%%%%%%%%%%%%%%%%%%%%%%%%%%%%%%%%%%%%%%%%
\begin{corollary}\label{nec1-SX-N}
Suppose $[g]$ is a conformal class for which $Y(M^n,[g])<0$, and assume that $g_e=\mu_e^{4/(n-2)}g$ is an extremal metric for $F^1_g$ with $\int_M\mu_e^{4/(n-2)}\;dv_g = 1$. Then  
\begin{equation}
 c_nR_g = \lambda_1(L_{g_e})\mu_e^{\frac{4}{n-2}}.
\end{equation}
In particular, the scalar curvature $R_g$ with respect to the background metric $g$ is negative everywhere.
\end{corollary}
%%%%%%%%%%%%%%%%%%%%%%%%%%%%%%%%%%%%%%%%%%%% 

As before, this put a restriction on the choice of background metric. In order to solve the maximization problem for $\Lambda_1(M,g)$ in this case, $R_g<0$ is a necessary condition. The proof of the following theorem is exactly as that of Theorem \ref{Existence1-SX-P}, and thus the majority of the arguments are omitted.

%%%%%%%%%%%%%%%%%%%%%%%%%%%%%%%%%%%%%%%%%%%%%%%
\begin{theorem}\label{Existence1-SX-N}
Let M be a closed n-dimensional manifold endowed with a conformal class [g] satisfying $Y(M^n,[g])<0$. If $R_g<0$, then the metric 
\begin{equation}g_{\text{\tiny max}} := \mu_{\text{\tiny max}}^{\frac{4}{n-2}}g = \left(\frac{R_g}{\int_M R_g\;dv_g}\right)g
\end{equation} 
solves the maximization problem for $\Lambda_1(M,g)$, i.e. $F_g^1(g_{\text{\tiny max}}) = \Lambda_1(M,g)$. In particular, $\Lambda_1(M,g)<0$.
\end{theorem}
%%%%%%%%%%%%%%%%%%%%%%%%%%%%%%%%%%%%%%%%%%%%%%

%%%%%%%%%%%%%%%%%%%%%%%%%%%%%%
\begin{proof}
Note that since $R_g$ is negative everywhere the conformal factor $\mu_{\text{\tiny max}}$ is well defined. As in Theorem \ref{Existence1-SX-P}, we get $F^1_g(g_{\text{\tiny max}}) = c_n\int_MR_g\;dv_g = \Lambda_1(M,g)$, and therefore $\Lambda_1(M,g)$ is strictly negative. 
\end{proof}
%%%%%%%%%%%%%%%%%%%%%%%%%%%%%%%%%%%%%%

%%%%%%%%%%%%%%%%%%%%%%%%%%%%%%%%%%%%%%%%%

\end{document}